\documentclass[11pt]{article}
\usepackage{amsmath,amssymb,amsthm}
\usepackage[latin1]{inputenc}
\usepackage[english]{babel}
\usepackage[normalem]{ulem}
\usepackage{xcolor}
\usepackage{indentfirst} 
\usepackage{cite}
\usepackage{xcolor}

\newtheorem{defn}{Definition}[section]
\newtheorem{thm}{Theorem}[section]
\newtheorem{prop}{Proposition}[section]
\newtheorem{cor}{Corollary}[section]
\newtheorem{lem}{Lemma}[section]
\newtheorem{rmk}{Remark}[section]
\newtheorem{exmp}{Example}[section]
\numberwithin{equation}{section}

\let \al=\alpha
\let \be=\beta
\let \var=\varphi
\let \vare=\varepsilon

\let \de=\delta
\let \th=\theta
\let \la=\lambda

\let \ga=\gamma
\let \Ga=\Gamma
\let \p=\partial
\let \q=\quad

\let \med=\medskip

\let \dps=\displaystyle
\let \ul=\underline

\let \ul=\underline
\let \ol=\overline

\newcommand{\R}{\mathbb{R}}
\newcommand{\N}{\mathbb{N}}
\newcommand{\C}{\mathbb{C}}

\DeclareMathOperator{\diag}{diag}

\begin{document}


\begin{center}
\textbf{\Large{Stability analysis of Richardson models with delay for confrontation between two countries }}
	\end{center}
	
	\
	
	\centerline{\bf Teresa Faria$^{a,}$\footnote{Corresponding author.}  and Anatoliy A.~Martynyuk$^b$} 
	
	\bigskip

	\centerline{\it $^a$Departamento de Ci\^encias Matem\'{a}ticas and CEMS.UL, Faculdade de Ci\^{e}ncias,}
	\centerline{\it Universidade de Lisboa, Campo Grande,
	1749-016 Lisboa, Portugal}
	\centerline{\it E-mail: teresa.faria@fc.ul.pt}
	
	\bigskip 	
	
\centerline{\it $^b$S.P.Timoshenko Institute of Mechanics of National Academy of Sciences of Ukraine,}
\centerline{\it 03057, Kyiv-57, Nesterov str., Ukraine}
\centerline{\it E-mail: journalndst@gmail.com}

\bigskip

\begin{abstract} This article proposes a non-autonomous mathematical   model  with delay for confrontation between two countries, and examines the stability  of its equilibrium state. Our criteria for stability take into account the influence of the factor of hostility between countries. For the autonomous case, the asymptotic stability is studied in a comprehensive way, and the Hopf bifurcations occurring  as the delay crosses some critical values  are described. For the  non-autonomous model, conditions ensuring the  global asymptotic stability  for both the linear approximation  and the nonlinear system are established.  The framework of special solutions for delay differential equations is also applied.
\end{abstract}

{\bf Keywords:} delay differential equations,   confrontation between two countries,  Richardson model, stability, asymptotic stability,  Hopf bifurcation, special solutions.

\med

{\bf Mathematical Subject Classifications (2020):}  34K17, 34K18, 34K20, 34K25, 91D10, 91F10. 
\\


	\section{Introduction}
	\setcounter{equation}{0}
	
	Lewis Richardson's classical  model of confrontation between two countries was created by Richardson in \cite{Rich} and developed by many authors (see \cite{Casp, EG, Smit, MDW} and the bibliography therein), for an era when arms races were measured in years.  In \cite{Sil,  BOMA}, Richardson's model was supplemented with an expanded ``hostility" vector, bringing it closer to reality and defining new conditions for equilibrium stability in a nuclear confrontation. A non-autonomous Richardson model with an expanded hostility vector was studied in \cite{Mar_S_S}.
	
	Today, in the era of hypersonic weapons, cybertechnology, and instantaneous  propaganda, the time lag ($\tau$) between one side's action and the other's conscious response becomes a critical factor in instability.
Richardson's time-lag model (perhaps proposed for the first  time by Hill \cite{HILL}) is the foundation of mathematical modelling for international conflict relations. Adding a time lag makes it even more realistic, as decisions on arms procurement and delivery never occur instantaneously.

In this paper, we will describe the confrontation between  two countries using the following mathematical model   with delay:
\begin{equation} \label{f1-1}
\begin{split}
\frac{dx}{dt}&=k(t)y(t-\tau)-a(t)x(t-\tau)+g(t, x(t), y(t), x(t-\tau), y(t-\tau)) \\
\frac{dy}{dt}&=l(t)x(t-\tau)-b(t)y(t-\tau)+h(t, y(t), x(t), y(t-\tau), x(t-\tau)),
\end{split}
\end{equation}
 where $x(t),y(t) \in \mathbb{R}_+$ measure the level of armament of countries $N_1$ and $N_2$ at time $t$, the coefficients $k(t)\ge 0, l(t)\ge 0$
characterize the threat between the countries, and  the coefficients $a(t)\ge a_{0}>0, b(t)\ge b_{0}>0$ characterize the expenses of each country to maintain its level of armament, for all $t \in \R$. 
The functions $k(t),a(t),l(t),b(t)$ and   $g(t,x,y), h(t,x,y)$ are continuous and nonnegative, and $k(t), l(t),a(t), b(t)$  bounded on $\R_{+}$.  The terms $g(t, x(t), y(t), x(t-\tau), y(t-\tau))$ and $h(t, y(t), x(t), x(t-\tau), y(t-\tau))$ characterize the "hostility" factor between countries $N_1$ and $N_2$, which depends, with a lag, on the level of confrontation between the countries and their armament. Surely, all the delays are not the same, as the reactions at time $t$ are not based solely on an observation at time $t-\tau$, but on a long period of observation, thus \eqref{f1-1} proposes a simplified version of the real world dynamics.

After observing the causes of two world wars, L. Richardson formulated a mathematical description of human motivation for war. He claimed that men are guided by ``their traditions, which are fixed, and their instincts, which are mechanical." As a result, in  \cite{Rich} he proposed describing the dynamics of the conflict between two countries using the system of equations
\begin{equation} \label{f1-3}
\begin{split}
\frac{dx}{dt}&=ky-ax+g \\
\frac{dy}{dt}&=lx-by+h,
\end{split}
\end{equation}
where $x(t)$ and $y(t)$ represent the armament expenditures (or armament levels) of two countries  $N_{1}$ and $N_{2}$ at time $t$,
$\frac{dx}{dt}$ and $\frac{dy}{dt}$ are the rate of change in expenditures over time for  $N_{1}$ and $N_{2}$, respectively, $k, l, a, b$ are positive constants and $g, h$ are real constants measuring the  hostility between the countries.

Perhaps the first author to introduce a time lag in the classical Richardson model  was Hill \cite{HILL}, who investigated the system of equations
\begin{equation}\label{Hill}
\begin{split}
\frac{dx}{dt}&=ky(t-\tau)-ax(t-\tau)+\bar{b}_1  \\
  \frac{dy}{dt}&=lx(t-\tau)-by(t-\tau)+\bar{b}_2.
  \end{split}
\end{equation}
Here, the constants $\bar{b}_{1}, \bar{b}_{2}$ characterize the boundary value of hostility between countries. We emphasize that  functional differential equations (FDEs) with delays often produce  models  which are more  realistic than ordinary differential equations (ODEs), as time-delays naturally appear to account for a variety of situations
 in population dynamics, control theory, epidemiology  and many  other scientific fields.

Note that in our proposed model  \eqref{f1-1} the hostility factors  are not constant. However, it is assumed that the hostility between countries $N_1$ and $N_2$ does not exceed a certain threshold level $b=(\bar{b}_{1}, \bar{b}_{2})^T$ over any finite time interval, regardless of the armament levels of the opposing countries.
Thus, for the components of the hostility vector function $F(t, x(t), y(t), x(t-\tau), y(t-\tau)):=(g(t, x(t), y(t), x(t-\tau), y(t-\tau)),  h(t, x(t), y(t), x(t-\tau), y(t-\tau)))$ there exist constants $\bar{b}_1>0$ and $\bar{b}_2>0$ such that
\begin{equation}\label{constraints}
\mid g(t, x(t), y(t), x(s), y(s))\mid \leq \bar{b}_1, \ \ \mid h(t, x(t), y(t), x(s), y(s))\mid \leq \bar{b}_2
\end{equation}
for all $t\ge 0, s = t-\tau$ and $x(t),y(t),x(s), y(s) \in {\R}_+$.


Unlike Richardson's classical autonomous linear model, our proposed generalized model takes into account three critical factors of modernity: time lag, nonlinearity, and the limitations of aggression (propaganda). While in the classical model the reaction is instantaneous, here the level of armament of one country, $x(t)$, depends on the level of armament of another country with a time lag, $y(t-\tau)$. An arms race is not only about hardware procurement (the linear part of the model), but also about informational and psychological conflict (the nonlinear part). We introduce the factor of ``limited hostility" through propaganda, postulating that even the resource of hatred has a saturation point. 

\begin{rmk} {\rm Propaganda through the ``unlimited hostility" media is one of the causes (not the least important) of wars between conflicting countries around the world. If propaganda $(g,h)$ has no ``brakes", i.e., it is unlimited, the system will inevitably follow a path of exponential growth toward war. Condition \eqref{constraints} essentially refects
 the hope that the human resource of hatred is finite, which offers a chance for system stabilization.}\end{rmk}

The major   goal of this mathematical study is to obtain clear stability criteria for equilibria of the   Richardson's model with delay \eqref{f1-1}, for both the autonomous and non-autonomous cases.

The delayed {\it autonomous} model will be first studied. The asymptotic stability or instability of its linearization at a positive equilibrium point is  analysed in a comprehensive way. With this analysis, we present a correction to Hill's result in \cite{HILL}, whose  research overestimated the stability margin of the ``world order". Our analysis shows that the critical time lag, at which the system looses stability and breaks into an irreversible race, is three times smaller than previously predicted by Hill. Then, we consider the nonlinear autonomous system: by
using the framework of central manifolds \cite{HaleLunel} and normal forms  \cite{FarMag},  we show that   a Hopf bifurcation occurs on the center manifold  at the equilibrium, as the delay crosses some critical values. Moreover,  these Hopf bifurcations are completely described.
In terms of real world interpretation, these results create a crisis-type forecasting tool: we can now determine whether an arms race will be ``soft" (stable cycles, where peace is possible as an average of fluctuations) or ``catastrophic" (subcritical bifurcation, where any error leads to direct conflict).

We then will  turn our attention to the {\it non-autonomous} arms race model with delay. The major focus is to provide stability and asymptotic stability criteria for the system with zero hostility over time, i.e., for the  (non-autonomous) linear approximation of \eqref{f1-1}. Several different techniques and approaches will be employed: Lyapunov functionals \cite{AML}, recent results on stability for linear  delay differential equations (DDEs) with ``pure diagonal" terms \cite{Berez_Diblik18,Faria22}, perturbation theory for linear DDEs \cite{Hale77}, and the framework of ``special solutions" \cite{AGP,Driver,GPDIE97}.
For the general nonlinear  system \eqref{f1-1},  conditions for its global asymptotic or exponential  stability will also be established by using a criterion in \cite{FOJDDE26}.

It is our believe that this proposed mathematical model and its study will help  understanding the nature of the Cold War and of modern hybrid conflicts.
An overall conclusion of our mathematical findings is the demonstration that there is a  threshold of reaction time, beyond which peace is impossible, even if both sides desire it.

\med

Some notation is now introduced. A  {\it solution} of system \eqref{f1-1} 
is a pair of continuous function $x:[t_0-\tau,a)\rightarrow \R$, $y:[t_0-\tau ,a)\rightarrow \R$  (with  $a>t_0$ or $a=\infty$) which is differentiable and satisfies
equations \eqref{f1-1} on the interval $[t_0,a)$, for some $t_0\in\R$. 
For  $t_0\in \R$ and   $\phi=(\phi_1,\phi_2): [-\tau, 0]\rightarrow \R^2$ continuous, the {\it solution} of system \eqref{f1-1} with {\it initial condition} $\phi$ at time $t_0$
is a solution $(x(t),y(t))$  of  \eqref{f1-1} defined on some interval $[t_0-\tau,a)$ such that $x(t+\th)=\phi_1(\th)$, $y(t+\th)=\phi_2(\th)$ for $\th\in [-\tau, 0]$. Under the conditions assumed here,  \eqref{f1-1} has a unique solution with initial condition $\phi$ at $t_0$,  denoted by  $(x(t,t_0,\phi),y(t,t_0,\phi))$, which is defined on the all interval $[t_{0}, \infty)$ \cite{HaleLunel}. 

The phase space for \eqref{f1-1}, 
unless  stated otherwise,  is the space $C:=C([-\tau,0];\R^2)$ of continuous functions from $[-\tau,0]$ to $\R^2$ with the supremum norm $ \|\phi\|=\max_{-\tau\le s\le 0}|\phi(s)|$, for some fixed norm $|\cdot |$ in $\R^2$. Without loss of generality, whenever it is convenient we may  fix the initial time to be $t_0=0$. The standard notation $(x_t,y_t)\in C$ denotes the segment of a solution on the interval $[t-\tau,t]$, defined as
$$x_t(\th)=x(t+\th),\ y_t(\th)=y(t+\th) \q {\rm for} \q \th\in [-\tau,0].$$

\med

The remainder of this article is organized as follows.  Section 2 is dedicated to the autonomous model: in the first subsection, the exponential asymptotic stability versus instability of the linearized equation at an equilibrium is established in detail, whereas the second subsection deals with the derivation of the direction and stability of the Hopf bifurcations occurring on a 2-dimensional central manifold  at the equilibrium, at each critical value of the delay $\tau$. The delay is taken as the bifurcation parameter.  In Section 3, we explore several different methods to derive sufficient conditions for the asymptotic stability of both the linear and nonlinear non-autonomous arms race model with delay.  In Section 4, the linear approximation of our model is considered as a particular case of a $p$-dimensional DDE of the form $x'(t)=A(t)x(t-\tau)$, where $A(t)$ is a $p\times p$ matrix of continuous functions,
 and the theory of ``special solutions"  for DDEs is applied to derive criteria for the asymptotic stability of such systems.  Some examples illustrate our results. Throughout the paper, some real world  interpretation of our mathematical findings will be provided. As a complement, a final section of  conclusions and political/social  interpretation of the results will be given.
	
	\section{Stability conditions in an  autonomous arms race model}
		\setcounter{equation}{0}
	
In this section, we shift the discussion of the problem to the plane of autonomous analysis. This allows us to find stationary points and assess the ``fundamental stability" of the world order. In fact, the transition from a system with variable coefficients to a system with constant parameters $k,a,l,b$ can be regarded as an average  of the political climate.

We consider  the autonomous  planar delay differential equation (DDE)
\begin{equation} \label{f2-1}
\begin{split}
\frac{dx}{dt}&=ky(t-\tau)-ax(t-\tau)+g( x(t-\tau), y(t-\tau)),\\
\frac{dy}{dt}&=lx(t-\tau)-by(t-\tau)+h( x(t-\tau), y(t-\tau)),
\end{split}
\end{equation}
where the constant model parameters $k, a, l, b$ as well as the time-delay $\tau$ are positive, the hostility functions $g,h$ are continuous and nonnegative, and all have the same meaning as for  system \eqref{f1-1}. More generally, the dependence on $x(t),y(t)$ can additionally be incorporated in $g,h$.


\subsection{The autonomous quasi-linear model}

In this section, we first address the stability of the quasi-linear model \eqref{Hill}.
Let $a,b,k,l, \bar{b}_{1},\bar{b}_{2}$ and $\tau$ be positive constants, and define
	$$A=\left [\begin{matrix}-a&k\\l&-b\\ \end{matrix}\right].$$
 If $\det A>0$, then the matrix $-A$ is a non-singular M-matrix, therefore its inverse $-A^{-1}$ is a non-negative matrix (see e.g. \cite{Fiedler}). In particular, we deduce that  the quasi-linear equation \eqref{Hill} has a unique positive equilibrium given  by the expression $P_{0}=-A^{-1}\bar b$, where $\bar b=(\bar{b}_{1},\bar{b}_{2})$.

%
%
%
%
%

The exponential asymptotic stability of the equilibrium $P_0$ is studied in the theorem below.

 \begin{thm}\label{thm2.1}
 With $a,b,l,b>0$ and $\det A=ab-kl>0$,  the positive equilibrium $P_0$ of \eqref{Hill} is exponentially asymptotically stable (EAS) if and only if 
\begin{equation}\label{0.2}\tau<\tau_- \, , \end{equation}
for $\tau_-$ given  by
\begin{equation}\label{0.1}\tau_{-}=\frac{\pi}{4\det A}\Big [(a+b) - \sqrt{(a+b)^2-4\det A}\Big].\end{equation}
 Moreover, there is a pair of roots of the characteristic equation of the form $\pm i\sigma\, (\sigma>0)$ if and only if
$\sigma=\sigma_n,\, \tau=\tau_n^{\pm},$
where $$\sigma_n=\frac{\pi}{2}+2n\pi,\q \tau_n^{\pm}= \sigma_n\left [ \frac{(a+b) \pm \sqrt{(a+b)^2-4\det A}}{2\det A}\right],\q n=0,1,2,\dots,$$
and these roots cross transversally  the imaginary axis to the right-hand-side as $\tau$ increases through each value $\tau_n^{\pm}$.
\end{thm}

\begin{proof} Write $X(t)=(x(t),y(t))$.  The linearization of system  \eqref{Hill}   about $P_0$ is  written  as
\begin{equation}\label{0.3}  X'(t)=AX(t-\tau),
\end{equation}  
with  characteristic equation given by
\begin{equation}\label{charac0}    
 \la^2+(a+b) \la e^{-\la\tau}+\det A e^{-2\la\tau}=0. \end{equation}
 For $\tau=0$, this system reduces to  the ODE  $X'=AX$, which is EAS, since the eigenvalues of $A$ are negative. For $\tau>0$, 
the only way the stability of \eqref{0.3} changes is when characteristic  roots  are crossing the imaginary axis $Re\, \la=0$ to the right (see e.g. the classical  monographs on FDEs \cite{HaleLunel,Kuang}).

In order to introduce the delay $\tau$ as a bifurcation parameter,
we normalize the delay by a time-scaling $t\mapsto \tau t$: in fact, with $\bar X(t)=X(\tau t)$, and dropping the bars for simplicity, system \eqref{0.3} now reads as 
$$X'(t)=\tau AX(t-1),$$ with characteristic equation
\begin{equation}\label{charac}       \la^2+a_0 \la e^{-\la}+c_0 e^{-2\la}=0, \end{equation}with $a_0=a_0(\tau):=\tau(a+b), c_0=c_0(\tau):=\tau^2 \det A.$ With $z=\la e^{\la}$, the above equation is equivalent to the quadratic equation $z^2+a_0z+c_0=0$, which has two real negative solutions given by
$$z_{\pm}=\frac{-a_0\pm\sqrt{a_0^2-4c_0}}{2}$$
(note that $a_0^2-4c_0>0$).
In particular, for $\la$ a root of \eqref{charac}  , then $\la e^{\la}\in (-\infty,0)$. Moreover, $0$ is not a characteristic root.

Let now $\la=\al+i\sigma\ (\al\ge 0, \sigma>0)$ be a solution of \eqref{charac}  . We start by looking at roots $\la=i\sigma$ with $\sigma>0$.

Since $Im\, (i\sigma e^{i\sigma})=0$ and $Re\, (i\sigma e^{i\sigma})<0$, we have
$\sin\sigma>0, \cos \sigma=0,$ and therefore 
$$\sigma=\sigma_n:=\frac{\pi}{2}+2n\pi, \q n=0,1,2,\dots.$$
Using  \eqref{charac}   with $\la=i\sigma_n$, we obtain the quadratic equation $-\sigma_n^2+\sigma_n a_0-c_0=0$. Solving it  for $\tau$, we derive the bifurcating values
$$\tau=\tau_n^{\pm}:= \sigma_n\left [ \frac{(a+b) \pm \sqrt{(a+b)^2-4\det A}}{2\det A}\right]$$
for $n=0,1,2,\dots$. In particular, the first pair of imaginary characteristic roots is $\pm i\sigma_0=\pm i\frac {\pi}{2}$ which appears for $\tau=\tau_0^-$, where
$\tau_0^-=\tau_-$ is as in \eqref{0.1}. Hence, the zero solution of  \eqref{0.3} is EAS for any $\tau<\tau_-$.

Next, rewrite \eqref{charac}        as  
\begin{equation}\label{0.5}     h(\la,\tau):=\la^2+\tau(a+b) \la e^{-\la}+\tau^2\det A\, e^{-2\la}=0,
\end{equation} so that
$$h(i\sigma_n,\tau_n^{\pm})=0,\q n=0,1,2\dots.$$
One easily sees that 
\begin{equation}\label{0.6}     
\begin{split}\frac{\p h}{\p \tau}(i\sigma_n,\tau_n^{\pm})&=\mp \sigma_n\sqrt{(a+b)^2-4 \det A}\\
\frac{\p h}{\p \la}(i\sigma_n,\tau_n^{\pm})&=\pm \tau_n^{\pm}\sqrt{(a+b)^2-4 \det A}+i \Big [2\sigma_n-\tau_n^{\pm}(a+b)\Big ],
\end{split}
\end{equation}  
therefore \eqref{0.5} defines a  smooth curve $\la(\tau)=\al(\tau)+i\sigma(\tau)$ of characteristic solutions, for $\tau$ close to each singularity point $\tau_n^{\pm}$ and such that 
$\la(\tau_n^{\pm})=i\sigma_n$, for all $n\in\N_0$.

Fix a  delay $\tau_n^{-}$ or $\tau_n^{+}$, $n\in \N_0$. Introducing $\la=\la(\tau)$ in \eqref{0.5}      and using \eqref{0.6}, we conclude that its derivative $\la'(\tau_n^{\pm})$ at $\tau_n^{\pm}$  satisfies
$$Re\, \la'(\tau_n^{\pm})=\frac{\sigma_n\tau_n^{\pm}[(a+b)^2-4 \det A]}{(\tau_n^{\pm})^2[(a+b)^2-4 \det A]+(C_n^{\pm})^2}>0,$$
where $C_n^{\pm}=2\sigma_n-\tau_n^{\pm}(a+b)$. Thus, the curve $\la(\tau)$ crosses transversally the imaginary axis to the right-hand-side  at each  point $\tau_n^{\pm}$, showing that 
there are no switches on the stability of $P_0$. In particular, $P_0$ is EAS if and only if $\tau<\tau_0^-$.
\end{proof}

\begin{rmk} \label{rmk2.1}  {\rm Under  conditions
\begin{equation}\label{hyp1Hill}
 ab > kl, \ \ \ a+b\leq \frac{3\pi}{\tau};
\end{equation}
\begin{equation}\label{hyp2Hill}
k\leq \frac{ab}{l}-\frac{3\pi(a+b)}{2\tau l}+\frac{9\pi^{2}}{4l\tau^{2}},
\end{equation}
Hill \cite{HILL} claimed that all  the roots of the characteristic equation \eqref{charac0}   satisfy the condition $Re \lambda < 0$, and therefore
 the solutions of system \eqref{Hill} exponentially approach the equilibrium state $P_{0}$, as $t\to\infty$. 
	If $\det A=ab-kl>0$,  one can write \eqref{hyp2Hill} as a quadratic inequality in $\tau$ and one easily sees that condition \eqref{hyp2Hill}  implies either $\tau<\tau_-^*$ or $\tau>\tau_+^*$
where $$\tau_{\pm}^*=\frac{3\pi}{4\det A}\Big [(a+b) \pm \sqrt{(a+b)^2-4\det A}\Big].$$	
Thus, since $\tau_+^*>\frac{3\pi}{a+b}$,  \eqref{hyp1Hill}  and \eqref{hyp2Hill}  together simplify as $\tau<\tau_-^*$.
In view of the statement of Theorem \ref{thm2.1}  proven above, it seems there is a minor mistake in reference \cite{HILL} regarding the conditions for the exponential asymptotic stability (EAS) of the equilibrium $P_0$ of \eqref{Hill}, since the above threshold $\tau_-^*$ should be divided by 3, so that it coincides with $\tau_{-}$ in \eqref{0.1}.
As seen,  this $\tau_{-}$ gives the first bifurcation point.}
\end{rmk}

\begin{rmk} \label{rmk2.2}{\rm  The transversality condition $Re\, \la' > 0$ means that stability will not return spontaneously with further increases in the delay. Once the  threshold $\tau_{-}$ is crossed, the system will only oscillate further. }
\end{rmk}

We now give some  interpretation of the above  theorem and formulas  in terms of an arms race:  condition $\det A>0$ is the ``survival basis": $ab>kl$  asserts that the product of the ``economic self-restraint" coefficients ($ab$) must be greater than the product of the ``mutual fear" coefficients ($kl$).
Here, we have proven that even if countries desire peace ($\det A>0$), the system can collapse only due to procrastination. Specifically, if the  time-delay $\tau$ (the time for intelligence, parliamentary debates, and factory construction) exceeds the threshold $\tau_{-}$, the stability of the equilibrium is lost. This may mean that delaying decisions does not solves the arms race, it only exacerbates it.  A slow response under conditions of mistrust transforms ``balance" into ``fluctuations." 

\subsection{Hopf bifurcation for the autonomous nonlinear model}

We now go back to  the autonomous   {\it nonlinear}  system 
\eqref{f2-1}, written in a general form as
\begin{equation}\label{0.7'} 
\begin{split}
\frac{dx}{dt}&=G(x(t-\tau),y(t-\tau))\\
\frac{dy}{dt}&=H(x(t-\tau),y(t-\tau)),    
\end{split}
\end{equation}  
with $G,F$ differentiable. Assume that there is a positive equilibrium $P_0$, i.e., a point $P_0>0$ satisfying $G(P_0)=H(P_0)=0$. Let $A=\left [\begin{matrix}-a&k\\l&-b\\ \end{matrix}\right]$ be the Jacobian matrix of $(G,H)$ at $P_0$.
Translating the equilibrium $P_0$ to the origin,  
the DDE  \eqref{0.7'} is transformed into
\begin{equation}\label{0.7} 
\begin{split}
\frac{dx}{dt}&=ky(t-\tau)-ax(t-\tau)+g(x(t-\tau),y(t-\tau))\\
\frac{dy}{dt}&=lx(t-\tau)-by(t-\tau)+h(x(t-\tau),y(t-\tau)),    
\end{split}
\end{equation} 
where
\begin{equation}\label{0.8} 
\left [\begin{matrix}G(P_0+(x,y))\\ H(P_0+(x,y))\end{matrix}\right]= A\left  [\begin{matrix}x\\ y \end{matrix}\right]+\left [\begin{matrix}g(x,y)\\ h(x,y) \end{matrix}\right],
\end{equation} 
for $(x,y)\in\R^2$,
with $g(0,0)=h(0,0)=0$ and   $Dg(0,0)=Dh(0,0)=0$.
 
 We are now interested in studying the occurrence of a Hopf bifurcation at $P_0$ for each $\tau_n^{\pm}$, for which  we follow the normal form procedure in  \cite{FarMag}. 
  Here, we completely describe the Hopf bifurcation when the nonlinear terms $(g,h)$ do not have quadratic terms in $(x,y)$.
To use the delay $\tau$ as the bifurcation parameter,  it is more convenient to proceed again  with the time-scaling $(\bar x(t),\bar y(t))=(x(\tau t ),y(\tau t))$. Dropping the bars for simplicity, system \eqref{0.7} is given  in matrix form by 
 \begin{equation}\label{0.9}  
 X'(t)=\tau AX(t-1)+\tau \left [\begin{matrix}g(X(t-1))\\ h(X(t-1))\end{matrix}\right],
 \end{equation} 
 where  $X(t)=(x(t),y(t))$.

\med

The following assumptions are now assumed:

\med

(H1)    $a,b,k,l>0$ and $ \det A>0$,

(H2) $F=(g,h)$ is of class $C^k$ with $k\ge 3$, and has no quadratic terms, so that the Taylor expansion at $(0,0)$ has the form
\begin{equation*}\label{F_3}
\begin{split}F(x,y)&=\frac{1}{3!}F_3(x,y)+{\rm o} (|(x,y)|^3),\\
 \end{split} 
  \end{equation*} 
where $F_3(x,y)$ is a homogeneous polynomial in the variables $(x,y)$ of degree 3.

\med

Computationally, the absence of quadratic terms simplifies the analysis of the Hopf bifurcation, since the bifurcation direction (supercritical or subcritical) will be determined directly by the cubic terms. Otherwise, we can still use the norm form algorithm in \cite{FarMag}, with longer computations to derive  the quadratic terms influence on the cubic terms of the ODE giving the flow on the center manifold.


 Consider \eqref{0.9}  in the phase space $C=C([-1,0];\R^2)$ with the supremum norm, also identified with its complexification $C([-1,0];\C^2)$. As mentioned in the Introduction, for $X:[\al,\infty)\to\C^2$, $X(t)=(x(t),y(t))$, we shall use the standard notation $X_t=(x_t,y_t)\in C$ to designate the segment   of  $X(t)$ defined on $[t-1,t]$, i.e.,
 $$X_t(\th)=X(t+\th)\q {\rm for}\q -1\le \th\le 0.$$
 
 Let $\sigma=\sigma_n,\, \tau=\tau_n^{\pm}$ be as in Theorem \ref{thm2.1}. At each singularity $\tau^*:=\tau_n^{\pm}$, we introduce the parameter
$$\mu=\tau-\tau_n^{\pm}, $$
so that \eqref{0.9}   becomes
 \begin{equation}\label{0.10}
 X'(t)=\tau^* AX(t-1)+\mu AX(t-1)+
 (\mu+\tau^*) F(X(t-1))
 \end{equation} 
where $F=(g,h)$ 
   with $F(0,0)=0, DF(0,0)=0, D^2F(0,0)=0$.

For clarity, below we first  study the local behaviour of solutions about $P_0$ for $\tau$ close to the  first bifurcation point $\tau^*:=\tau_0^-=\tau_{-}$ as in \eqref{0.1}, although later   we extend the results to all bifurcation points $\tau=\tau_n^{\pm}$. We express $\tau^*=\tau_0^-$ as
$$\tau^*=\sigma_0\rho^*$$
for $\sigma_0=\pi/2$ and
$$ \rho^*= \frac{(a+b) - \sqrt{(a+b)^2-4\det A}}{2\det A}.$$

For system \eqref{0.10}, define $L_0(\var)=A\var(-1)$ for $\var=(\var_1,\var_2)\in C,$
and let $\Lambda:=\{i\sigma_0,-i\sigma_0\}$. Using the formal adjoint
theory for FDEs in \cite{HaleLunel},  the phase space $C$ is
 decomposed by $\Lambda $ as
$$C=P\oplus Q,$$
where $P$ is the  eigenspace for the linear equation $X'(t)=\tau^* L_0(X_t)$ associated with the eigenvalues in $\Lambda$. 
Consider the ``formal duality" $(\cdot ,\cdot )$ associated with this linear 
equation, given  by  the bilinear form 
$$(\psi ,\phi)=\psi(0)\phi(0)-\tau^*L_0\left (\int_0^\th \psi(\xi-\th)\var(\xi)\, d\xi\right),$$
where above we abused the notation, and wrote $L_0(\var(\th))$ instead of $L_0(\var)$, for $\var\in C$. Let $\Phi=[\phi_1\  \phi_2]$ and  $\Psi=\left [\begin{matrix} \psi_1\\ \psi_2\end{matrix}\right]$ be  bases for $P$ and
its dual space $P^*$, respectively, where $P^*$ is the  eigenspace associated with the eigenvalues $\pm i\sigma_n$ of the adjoint equation
 defined as $Y'(s)=-\tau^* Y(s+1)A$ for $Y(t)^T\in\C^2$;
suppose in addition  that the bases $\Phi,\Psi$ are  normalized so  that $(\Psi ,\Phi )=I$. Moreover, note that
$\dot \Phi=B\Phi, -\dot\Psi=\Psi B$, for $B$ the diagonal matrix $B=\diag(i\sigma_0,-i\sigma_0)$.

%

 In matrix form, $\Phi,\Psi$ are given by the following $2\times 2$-matrix functions:
$$\Phi (\th)=\Big [e^{i \frac{\pi}{2}\th} v\ \ e^{-i \frac{\pi}{2}\th}\ol{v}\Big ],\q
\Psi (s)= \left [\begin{matrix}e^{-i\frac{\pi}{2} s}u^T\\ e^{i\frac{\pi}{2}s} \ol{u}^T \end{matrix}\right],\q {\rm for}\q \th\in [-1,0],s\in [0,1],$$
where the bar means complex conjugation, $u^T$ is the transpose of $u$
and  $v=(v_1,v_2)\in \C^2, u=(u_1,u_2)\in \C^2$ are vectors determined by:

$\bullet$  $\phi_1(t)=e^{i \frac{\pi}{2}t }v$ is a solution of the linear equation $X'(t)=\tau^*AX(t-1)$ (equivalently, $v$ satisfies
$\tau^*L_0( e^{i \frac{\pi}{2}\cdot }v)=i \frac{\pi}{2} v$);

$\bullet$ $\psi_1(s)=e^{-i\frac{\pi}{2} s}u^T$ is a solution of the adjoint equation 
 $Y'(s)=-\tau^* Y(s+1)A$ (where $Y(t)^T\in\C^2$);

$\bullet$ $(\psi_1,\phi_1)=1$
 for the above formal duality, which is equivalent to $u^Tv+i\sigma_0 u^Tv=1$.\\
 Some computations  lead to
$$v_2=\frac{\rho^*a-1}{\rho^*k}v_1,\q  u_2=\frac{\rho^*a-1}{\rho^*l}u_1$$
and
$$(1+i\sigma_0)\left (v_1+\frac{(\rho^*a-1)^2}{(\rho^*)^2 kl}\right)u_1=1,$$
where, of course, one may choose $v_1=1$, so that $v=\left [\begin{matrix}1\\ \frac{\rho^*a-1}{\rho^*k}\end{matrix}\right]$.
In particular, $v$ is a real vector.

In what follows, we 
 refer to\cite{FarMag} for results and
explanations of the several notations
involved.

Enlarging the phase space $C$ by considering the space $BC=\{ \var :[-1,0]\to \C^2:
\var \ {\rm is\ continuous}$  on $[-1,0)$ and $\exists \lim_{\theta \to 0^-}\var (\theta )\} $, we extend the natural projection from $C=P\oplus Q$ onto $P$ to a projection $\pi:BC\to P$, and
use it to construct the decomposition $BC=P\oplus Ker\, \pi$; thus, we write $z_t=\Phi x(t)+y_t$, where now $x,y$ have the following meaning: $x(t)=(x_1(t),x_2(t))\in \C^2$ and $y_t\in Q^1=Ker\, \pi\cap C^1$. We also decompose \eqref{0.10} according to $BC=P\oplus Ker\, \pi$, as
\begin{equation}\label{Eq_P+Q}\begin{cases}
\dot x&=Bx+\Psi (0)F_0(\Phi x+y,\mu  )\\
\dot y&={\cal A}_{Q^1}y+(I-\pi )X_0F_0(\Phi x+y,\mu  ),
\end{cases}
\end{equation}
where $F_0(\var,\mu)=\mu   A\var(-1)+(\tau^*+\mu) F(\var)$, ${\cal A}_{Q^1}=\tilde {\cal A}_{|_{Q^1}}$ and $\tilde{\cal A}$ is the extension to $C^1$ of the infinitesimal generator ${\cal A}$ for the $C_0$-semigroup generated by the semiflow of $X'(t)=\tau^* AX(t-1)$.

Consider the Taylor formula
$$
\Psi (0) F_0(\Phi x+y,\mu  )=\frac{1}{2}f_2^1(x,y,\mu  )+\frac{1}{3!}f_3^1(x,y,\mu  )
+h.o.t. ,
$$
where $f_j^1(x,y,\mu  )$
are homogeneous polynomials in
$(x,y,\mu  )$ of degree $j \, (j=2,3)$  with coefficients in $\C^2$
and $h.o.t.$ stands for higher order terms. In virtue of (H2), this means that
\begin{equation}\label{0.12} 
\begin{split}
f_2^1(x,y,\mu)&=2\mu\, \Psi (0)A(\Phi(-1)x+y(-1)),\\
f_3^1(x,y,\mu)&=\tau^*\Psi (0) F_3(\Phi(-1)x+y(-1)).
\end{split}
\end{equation}

Then, the flow  on the 2-dimensional
center manifold  for  \eqref{0.10}  at $x=0$ and  $\mu  =0$ satisfies  an ODE of the form
\begin{equation}\label{nf1} 
\dot x=Bx+\frac{1}{2}g_2^1(x,0,\mu  )+\frac{1}{3!}g_3^1(x,0,\mu  )
+h.o.t.,
\end{equation} 
where $g_2^1,g_3^1$ are the second and third order terms in $(x,\mu  )$, respectively, which will be expressed below in terms of the original coefficients by using the normal form method.

Note that $\Phi(-1)x=i \left [\begin{matrix}-v& v\end{matrix}\right]x=i(-x_1+x_2)v$ for $x=(x_1,x_2)$ and $\Psi (0)= \left [\begin{matrix}u^T \\  \ol{u}^T \end{matrix}\right] $. Since $Av=-\frac{1}{\rho^*}v$ and $(1+i\sigma_0) u^Tv=1$, we have
\begin{equation}\label{0.12'} 
\begin{split}
f_2^1(x,0,\mu)&=2i\mu \, (-x_1+x_2)  \Psi (0) Av\\
&= \frac{2i\mu}{\rho^*} \, (x_1-x_2)   \left [\begin{matrix}u^T v\\  \ol{u}^Tv \end{matrix}\right] \\
&=  \frac{2i\mu}{\rho^*(1+\sigma_0^2)} \, (x_1-x_2)   \left [\begin{matrix}1-i\sigma_0\\  1+i\sigma_0\end{matrix}\right]. \\
\end{split}
\end{equation}

Always following the algorithm given in \cite{FarMag}, we  recall that
$$
\frac{1}{2}g_2^1(x,0,\mu  )=\frac{1}{2}Proj_{Ker (M_2^1)}f_2^1(x,0,\mu  ),\ {\rm where}\
Ker (M_2^1)=span \bigg  \{ \left(\begin{matrix}x_1\mu  \\ 0\end{matrix}\right),\left( \begin{matrix}0\\ x_2\mu  \end{matrix}\right)\bigg \}.
$$

Together with \eqref{0.12'},  this yields that the second order terms in $(x,\mu )$ of the normal form \eqref{nf1} on the center manifold are given by
\begin{equation}\label{g_2} 
\frac{1}{2}g_2^1(x,0,\mu  )=\left[\begin{matrix}A_1x_1\mu  \\ \overline A_1x_2\mu\end{matrix}\right]\q {\rm with}\q A_1=\frac{i (1-i\sigma_0)}{\rho^*(1+\sigma_0^2)}=\frac{\pi/2+i}{\rho^*(1+\pi^2/4)}.
\end{equation}

 We now compute  the cubic terms $g_3^1(x,0,\mu  )$ in \eqref{nf1}.  
After the change of variables which transforms the quadratic terms $f_2^1(x,y,\mu  )$
of the first equation in \eqref{Eq_P+Q}  into $g_2^1(x,y,\mu  )$,  in this situation   the
coefficients of third order
at $y=0, \mu  =0$ are indeed the same, i.e., they are  still given by $\frac{1}{3!}f_3^1(x,0,0)$ (because $D^2F(0,0)=0$
implies that $f_2^1(x,y,0)=0$).
On the other hand, this yields (see \cite{FarMag})
$$\frac{1}{3!}g_3^1(x,0,\mu  )=\frac{1}{3!}Proj_{Ker (M_3^1)} f_3^1(x,0,\mu  ),$$
where
$$Ker (M_3^1)=span \bigg \{\left(\begin{matrix}x_1^2x_2\\ 0\end{matrix}\right), \left(\begin{matrix}x_1\mu  ^2\\ 0\end{matrix}\right),
\left(\begin{matrix}0\\ x_1x_2^2\end{matrix}\right),\left(\begin{matrix}0\\ x_2\mu  ^2\end{matrix}\right) \bigg\}.
$$
The terms $O(|x|\mu  ^2)$ are irrelevant to determine a generic Hopf bifurcation, 
 hence we write
$$
\frac{1}{3!}g_3^1(x,0,\mu  )=
\frac{1}{3!} Proj_S f_3^1(x,0,0)+O(|x|\mu  ^2),\ {\rm for}\ 
S:=span \Big \{\left(\begin{matrix}x_1^2x_2\\ 0\end{matrix}\right),
\left(\begin{matrix}0\\ x_1x_2^2\end{matrix}\right) \Big\} .
$$
From \eqref{0.12}, observe that
\begin{equation}\label{f_3} 
\begin{split}
f_3^1(x,0,0)&=\tau^*\Psi (0)F_3(i(-x_1+x_2)v) \\
&=-i\tau^*(-x_1+x_2)^3 \left [\begin{matrix}u^T \\  \ol{u}^T \end{matrix}\right] 
\left [\begin{matrix}g_{30}+3g_{21}v_2+3g_{12}v_2^2+g_{03}v_2^3\\  
h_{30}+3h_{21}v_2+3h_{12}v_2^2+h_{03}v_2^3 \end{matrix}\right],
\end{split}
\end{equation}    
where $g_{ij}=\frac{\p^3 g}{\p x^i\p x^j}(0,0),  h_{ij}=\frac{\p^3 h}{\p x^i\p x^j}(0,0)$ for $0\le i,j\le 3, i+j=3$. From  \eqref{f_3}, easy computations allow to deduce that
\begin{equation}\label{g_3} \frac{1}{3!}g_3^1(x,0,0)=
\left[\begin{matrix}A_2x_1^2x_2\\ \overline A_2x_1x_2^2\end{matrix}\right],\end{equation}  
  with 
\begin{equation}\label{A_2} 
\begin{split}
A_2
=&\frac{-i\tau^*}{2}\Big [
u_1 (g_{30}+3g_{21}v_2+3g_{12}v_2^2+g_{03}v_2^3)\\&+u_2 (h_{30}+3h_{21}v_2+3h_{12}v_2^2+h_{03}v_2^3)\Big]\\
= &\frac{-i\tau^*}{2}u_1\Big [(g_{30}+3g_{21}v_2+3g_{12}v_2^2+g_{03}v_2^3)\\
&+ \Big (\frac{\rho^*a-1}{\rho^*l}\Big) (h_{30}+3h_{21}v_2+3h_{12}v_2^2+h_{03}v_2^3)\Big].
\end{split}
\end{equation}

%
%
From this normal form computations, we deduce that the normal form \eqref{nf1} on the center manifold
for \eqref{Eq_P+Q}  at $x=0,\mu=0$ has the terms $g_2^1,g_3^1$ given by \eqref{g_2},\,\eqref{g_3},
with the coefficients $A_1,A_2$ explicitly obtained in terms of the original equation 
\eqref{0.7'}, also given by \eqref{0.7}.
Changing to real coordinates $w$, where $x_1=w_1-iw_2,x_2=w_1+iw_2$, and afterwards to
 polar coordinates $(r , \xi )$,
$w_1=r \cos \xi,w_2=r \sin \xi$,   the normal form
\eqref{nf1}  becomes
\begin{equation}\label{nf_polar}
\begin{cases}
\dot r &=K_1\mu  r +K_2 r^3+O(\mu  ^2r +|(r ,\mu  )|^4)\\
\dot \xi &=-\sigma_0+O(|(r ,\mu  )|),\end{cases}
\end{equation}
with 
$K_1:=Re\, A_1,\ K_2:=Re\, A_2.$ To determine a generic Hopf bifurcation next to the origin, only the signs of $K_1,K_2$ are relevant. Clearly,
$$K_1=\frac{\pi/2}{\rho^*(1+\pi^2/4)}>0$$
and
$$ sign\, K_2
=- sign \Big [g_{30}+3g_{21}v_2+3g_{12}v_2^2+g_{03}v_2^3+ \Big (\frac{\rho^*a-1}{\rho^*l}\Big) (h_{30}+3h_{21}v_2+3h_{12}v_2^2+h_{03}v_2^3)\Big],
$$
where $v_2=\frac{\rho^*a-1}{\rho^*k}$.

With $K_2\not =0$,  a  generic Hopf bifurcation occurs. Recall that $K_1>0$.
It is well know (see e.g. \cite{ChowHale}) that   when $K_2<0$
  the Hopf birfurcation is supercritical and the  nontrivial periodic orbits about $P_0$ are stable, whereas $K_2>0$ implies that the Hopf birfurcation is subcritical with unstable nontrivial periodic orbits.
  
  \med
  
  We now extend the above computations and results  to all bifurcation points $\tau_n^{\pm}$, omitting some details. Fix  a point $\tau^*=\tau_n^{\pm}$ (i.e., $\tau^*=\tau_n^{-}$ or $\tau^*=\tau_n^{+}$), for some $n\in\N_0$, and further assume that:
  
  \med

(H3) $\tau_n^{\pm}\ne \tau_p^-$ and $\tau_n^{\pm}\ne \tau_p^+$ for all $p\ne n$.

\med

This hypothesis implies that, at the singularity point $\tau_n^{\pm}$, all other characteristic roots different from $\pm i\sigma_n$ have non-zero real parts. 
 Next, express $\tau_n^{\pm}$    as
$$\tau_n^{\pm}=\sigma_n\rho^\pm$$
for $\sigma_n=\pi/2+2n\pi$ and
$$\rho^\pm:= \frac{(a+b) \pm \sqrt{(a+b)^2-4\det A}}{2\det A}.$$
Take $\Lambda:=\{i\sigma_n,-i\sigma_n\}$ and normalized bases $\Phi$ and $\Psi$  for $P$ and $P^*$,
$$\Phi (\th)=\Big [e^{i \sigma_n\th} v\ \ e^{-i \sigma_n\th}\ol{v}\Big ],\q
\Psi (s)= \left [\begin{matrix}e^{-i\sigma_n s}u^T\\ e^{i\sigma_n s} \ol{u}^T \end{matrix}\right],\q {\rm for}\q \th\in [-1,0],s\in [0,1].$$
One can easily check that the vectors $v$ and $u$ and the coefficients $K_1,K_2$ are exactly as the ones given by the previous formulas, with $\rho^*, \tau^*,\sigma_0$ replaced by $\rho^\pm, \tau_n^\pm,\sigma_n$, respectively.

%


 For the original equation \eqref{0.7'} and with the previos notations, the conclusions are summarized in the following  result:

 \begin{thm}\label{thm2.2}
 For each $n\in\N_0$, fix a bifurcating point $\tau=\tau_{n}^\pm$. Assume (H1),~(H2) and (H3). The flow on the center manifold of \eqref{0.7'}  for $(x,y)=P_0,\tau=\tau_{n}^\pm$ is given in polar coordinates $(r,\xi)$ by
  \begin{equation}\label{nf_polar-n}
\begin{cases}
\dot r &=K_1\mu  r +K_2 r^3+O(\mu  ^2r +|(r ,\mu  )|^4)\\
\dot \xi &=-\sigma_n+O(|(r ,\mu  )|),\end{cases}
\end{equation}
with $\dps K_1=\frac{\sigma_n}{\rho^\pm(1+\sigma_n^2)}$ and
 \begin{equation}\label{K2_n}
 \begin{split} K_2&=\frac{-\tau_n^\pm \sigma_n}{1+\sigma_n^2} \left (1+\frac{(\rho^\pm a-1)^2}{(\rho^\pm)^2 kl}\right)^{-1} 
\Big [g_{30}+3g_{21}v_2+3g_{12}v_2^2+g_{03}v_2^3\\
&+ \Big (\frac{\rho^\pm a-1}{\rho^\pm l}\Big) (h_{30}+3h_{21}v_2+3h_{12}v_2^2+h_{03}v_2^3)\Big],
\end{split}
\end{equation}
where $v_2=\frac{\rho^\pm a-1}{\rho^\pm k}$. If $K_2<0$ (respectively $K_2>0$),
then there is a Hopf bifurcation occuring on this center manifold,  
 which is supercritical (respectively,  subcritical) with
  non-trivial
periodic orbits stable (respectively, unstable) on the center manifold.\end{thm}

\begin{exmp} {\rm Consider \eqref{0.7'} with 
$A=\left [\begin{matrix}-2&3\\ \frac{1}{4}&-1 \end{matrix}\right]$. With the above notations,
$\det A=\frac{5}{4}, \rho_-=\frac{2}{5}, \rho_+=2$. For this case,
\[
\begin{split}
&\sigma_0=\frac{\pi}{2}, \tau_0^-=\frac{\pi}{5}, \tau_0^+=\pi,\q  \sigma_1= \frac{5\pi}{2},  \tau_1^-=\pi, \tau_1^+=5\pi ,  \q \dots \q ,\\
&\sigma_n= \frac{(4n+1)\pi}{2},  \tau_1^-=\frac{(4n+1)\pi}{5}\pi, \tau_1^+=(4n+1)\pi,\q n\in\N.
\end{split}
\]
Since $ \tau_0^+=\tau_1^-$, hypothesis (H3) is not satisfied for $\tau=\pi$, and
Theorem \ref{thm2.2} cannot be invoked. In fact, for $\tau=\pi$ the characteristic equation \eqref{charac0}   has two pairs of characteristic roots, $\pm i\frac{\pi}{2}$ and $\pm i\frac{5\pi}{2}$, and a double-Hopf bifurcation occurs on a  4-dimensional  center manifold.

Fix now $\tau_-=\tau_0^-=\frac{\pi}{5}$. In this case,  $v=\left [\begin{matrix}1\\ -\frac{1}{6}\end{matrix}\right]$ and ${\rm sign}\,  K_2={\rm sign}\,\big [-(g_{30}-\frac{1}{2}g_{21}+\frac{1}{12}g_{12}-\frac{1}{6^3} g_{03})
+2 (h_{30}-\frac{1}{2}h_{21}+\frac{1}{12}h_{12}-\frac{1}{6^3} h_{03})\big]$. If $K_2<0$, there a supercritical Hopf bifurcation at $(x,y)=P_0,\tau=\frac{\pi}{5}$, with the bifurcating non-trivial
periodic solutions stable.
}
\end{exmp}

\begin{rmk}\label{rmk2.2}{\rm  We give an  interpretation of the threshold delay as well as of  the Hopf bifurcation in terms of an arms race.
The proof that the characteristic roots intersect the imaginary axis to the right as $\tau$ increases signifies the transition of system \eqref{0.7}  to self-oscillation.  For $\tau>\tau_{-}$, the arms race ceases to converge toward point $P_{0}$, and the system enters an endless  state of chronic instability, where periods of lull are merely part of the cycle of accumulating strength for a new surge.
Note also that the fact that $K_{1}$ and $K_{2}$ retain their structure for all subsequent bifurcation points indicates the fractal nature of instability: even if the system miraculously passes the first risk point $\tau_{-}=\tau_{0}^{-}$, it will inevitably encounter the next $\tau_{0}^{+}$ or $\tau_{1}^{-}$, with each new critical value of the delay adding new frequencies to the ``fear oscillations", making world politics increasingly complex and unpredictable.}
\end{rmk}




\section{Stability Conditions for the non-autonomous  model}
\setcounter{equation}{0}

We major aim of this section is to provide stability and asymptotic stability conditions for the system with zero hostility over time, using several different techniques and approaches. For the {\it nonlinear} system \eqref{f1-1}, a criterion for its asymptotic stability will also be established.

\subsection{Stability Conditions under zero hostility over time}

Consider system \eqref{f1-1} 
 with zero hostility, $g(t, x(t), y(t), x(t-\tau), y(t-\tau))=0$ and $h(t, x(t), y(t), x(t-\tau), y(t-\tau))=0$ for any  $t \in \R_{+}$, so that \eqref{f1-1} reads as  the non-autonomous linear DDE
\begin{equation}\label{f2-3}
\frac{dx}{dt}=k(t)y(t-\tau)-a(t)x(t-\tau) , \ \ \frac{dy}{dt}=l(t)x(t-\tau)-b(t)y(t-\tau),
\end{equation}
also written in matrix form as 
\begin{equation}\label{f2-4}X'(t)=A(t)X(t-\tau),\end{equation}
for 
$X(t)=(x(t),y(t))$ and
\begin{equation*}\label{A(t)}
A(t)=\left [\begin{matrix}-a(t)&k(t)\\l(t)&-b(t)\\ \end{matrix}\right].
\end{equation*}
The matrix $A(t)$ is designated the {\it armament matrix}.

If there are no delays in \eqref{f2-4}, the stability of the system can be easily derived by considering
the Lyapunov function  $V(x,y)=\|(x,y)\|^2=x^{2}+y^{2}$ and its total derivative along  solutions.

\begin{prop}\label{prop3.1}  Consider the ODE system without delay $X'(t)=A(t)X(t)$, where the  parameters $k(t), a(t), l(t), b(t)$ are continuous and $a(t)\ge 0,b(t)\ge 0$. If
\begin{equation}\label{Lya1}
a(t)b(t)\ge \frac{(k(t)+l(t))^{2}}{4},\q t\ge t_0,
\end{equation}
 then the equilibrium state $x=y=0$ is uniformly stable on $[t_0,\infty)$. Moreover, if there exist positive constants $\al,\beta$ and $t_1\ge t_0$ such that
\begin{equation}\label{Lya2}
a(t)+b(t)\ge \al\q {\rm and}\q
 a(t)b(t)-\frac{(k(t)+l(t))^{2}}{4}> \beta, \q t\ge t_1,
\end{equation}
 then the equilibrium state $x=y=0$ of system $X'(t)=A(t)X(t)$ is asymptotically stable.\end{prop}

\begin{proof}
Under  condition \eqref{Lya1},  for  $z\in\R^2$, the total derivative $V(x,y)=x^{2}+y^{2}$ along solutions
$x(t)=x(t,t_{0}, z),y(t)=y(t,t_{0}, z)$ is
\begin{equation}
\begin{split}
\frac{d}{dt}V(x,y)  &= 2x(t)(k(t)y(t)-a(t)x(t))+2y(t)(l(t)x(t)-b(t)y(t))\\
&\le -2\Big(\sqrt{a(t)}x(t)-\sqrt{b(t)}y(t)\Big)^2\le 0.
\end{split}
\end{equation}
Hence, the inequality $V(x(t,t_{0}, z), y(t,t_{0}, z))\leq V(0,0)=0$ holds for all $t\geq t_{0}$. For $\vare>0$, 
if $\|z\|^2\le \delta$, for some
$\delta=\delta(\varepsilon)>0$, then
$x^{2}(t,t_{0}, z)+ y^{2}(t,t_{0}, z)\leq \varepsilon$, implying that $(0,0)$ is uniformly stable.

Assume now \eqref{Lya2}. For  the total derivative of $V(x,y)$ along solutions,  one easily derives that there exists a function $\ga(t)$, with $\ga(t)\ge \vare$ for some $\vare>0$,
\begin{equation}\label{Lya3}
\frac{d}{dt}V(x(t,t_{0}, z), y(t,t_{0}, z)) < -\gamma(t) V(x(t,t_{0}, z), y(t,t_{0}, z))
\end{equation}
for all $t\geq t_{0}$ and $z\in \R^2$. From estimate \eqref{Lya3}, it follows that
\begin{equation}
V(x(t,t_{0}, z), y(t,t_{0}, z)) < V(x_{0}, y_{0})e^{-\int_{t_{0}}^{t}\gamma(s)ds}
\end{equation}
for all $t\geq t_{1}$ and $z=(x_{0}, y_{0})\in \R^2$. Consequently, $x^{2}(t,t_{0}, z)+ y^{2}(t,t_{0}, z)\rightarrow 0$ as $t\rightarrow \infty$. This proves the proposition.
\end{proof}

We emphasize that there are other criteria for the stability of  $n$-dimensional linear ODEs $x'(t)=A(t)x(t)$, based on e.g. the sum or maximum norms on $\R^n$. See e.g. the monographs by Coppel \cite {Coppel} or
Martynyuk \cite{Martyn}.

\med

To address the stability of the zero solution of \eqref{f2-4}, we recall that the construction of Lyapunov functionals for DDEs is in general a hard task, mostly in the case of pure diagonal delays. The more general method of Lyapunov-Krasovskii has also been applied with success, see e.g. \cite{TAB}. However, such methods   are not often used when only pure diagonal delays are present, since they mostly rely on some kind of dominance by  non-delayed diagonal terms which somehow control the disturbing effect of the delays.

On the other hand, 
obtaining criteria of {\it absolute} stability -- i.e., stability criteria which are valid for all values of  the delays --, is not feasible here, since Theorem \ref{thm2.1} asserts that   for the particular case of an autonomous system  $X'(t)=AX(t-\tau)$ the absolute stability is never possible, and  that the stability does depend on the size of the delay. 

Below, we apply to the arms race model some stability results for linear non-autonomous DDEs which have been established in recent literature.

\begin{thm}\label{thm3.1} Consider system \eqref{f2-3}, with $a(t),b(t),k(t),l(t)$ continuous and nonnegative  functions on $[t_0,\infty)$, and $a(t),b(t)$ positive. Define
 $$c_1(t)=a(t)\int_{t-\tau}^t (a(u)+k(u))\, du,\q c_2(t)=b(t)\int_{t-\tau}^t (b(u)+l(u))\, du.$$
Suppose that one of the following conditions is satisfied:\\
(a) $a(t),b(t),k(t),l(t)$ are  bounded and there exists a positive constant vector $v=(v_1,v_2)$ such that
\begin{equation}\label{Mv}\liminf_{t\to\infty} \hat M(t)v>0,\end{equation}
where
$$\hat M(t)=\left [\begin{matrix}a(t)-c_1(t)&-k(t)\\-l(t)&b(t)-c_2(t)\\ \end{matrix}\right];$$
(b) $a(t),b(t)$ are  bounded from below by positive constants and there exist  $\al>1$ and a positive constant vector $v=(v_1,v_2)$ such that 
\begin{equation}\label{Mv2}\diag (a(t),b(t))v\ge \al \left [\begin{matrix}c_1(t)&k(t)\\l(t)&c_2(t)\\ \end{matrix}\right]v,\q {\rm for}\q t\gg 1.\end{equation}
Then system \eqref{f2-3} is EAS.
\end{thm}	

\begin{proof} Condition \eqref{Mv} is equivalent to  $\hat M(t)v\ge u$ for $t$ sufficiently large and some positive vector $u=(u_1,u_2)$, thus hypothesis (H4) in \cite{Faria22} is satisfied.
Also, \eqref{Mv2} implies that hypothesis (H5)  in \cite{Faria22} holds.  Consequently, the conclusions in  (a) and (b) follow immediately from \cite[Theorem 4.1 and Corollary 4.2]{Faria22}. \end{proof}

Remark that, with the assumed limitation \eqref{constraints}, $a(t),b(t),k(t),l(t)$ are  always bounded, so the criterion in (a) can be applied.
In particular with $v=(1,1)$ in (a), we obtain:
\begin{cor}\label{cor3.1} Consider system \eqref{f2-3}, with $a(t),b(t),k(t),l(t) $ continuous, nonnegative and bounded  functions on $[t_0,\infty)$, and $a(t),b(t)$ positive. If
\[
\begin{split}
    &\limsup_{t\to\infty}\left [\int_{t-\tau}^t (a(u)+k(u))\, du+\frac{k(t)}{a(t)}\right]<1\\
&\limsup_{t\to\infty}\left [\int_{t-\tau}^t (b(u)+l(u))\, du+\frac{l(t)}{b(t)}\right]<1,\\
\end{split}
\]
then system \eqref{f2-3} is EAS.
\end{cor}

\begin{rmk}\label{rmk3.1} {\rm The criteria above are  based on  \cite{Faria22}, where the stability for  general linear non-autonomous DDEs with either finite or infinite delays was studied. Hence, they apply to the more general situation of a time-variang delay $\tau(t)$ (or several delays $\tau_{ij}(t),\, i,j=1,2$), instead of a constant delay $\tau$. We also note that Corollary \ref{cor3.1} is similar to criteria established  in
\cite[Corollary 3]{Berez_Diblik18}, however the additional condition in \cite{Berez_Diblik18} of having $a(t),b(t)$ bounded from below by  positive constants is not required in Theorem \ref{thm3.1}.(a). 
}\end{rmk}

The research in  \cite{Faria22} was generalized to nonlinear systems of FDEs (with possibly infinite delay) by Faria and Oliveira \cite{FOJDDE26}. Applied to system
 \eqref{f1-1},  the results in \cite{FOJDDE26} lead to the following criteria:

	\begin{thm}\label{thm3.2} Consider system \eqref{f1-1},  with $a(t),b(t),k(t),l(t)$ continuous and nonnegative  functions on $[t_0,\infty)$, and $a(t),b(t)$ positive.
Assume that:

(i) $\int_0^\infty a(t)\, dt=\infty$ and  $\int_0^\infty b(t)\, dt=\infty$;

(ii) there are continuous functions $G,H:\R_+\times \R^4\to \R_+$ such that, for all $(t,x,y,u,v)\in R_+\times \R^4$, it holds
$$|g(t,x,y,u,v)|\le G(t)|(x,y,u,v)|,\q |h(t,x,y,u,v)|\le H(t)|(x,y,u,v)|,$$
where $|\cdot|$ is the maximum norm in $\R^4$;

(iii) $\limsup_{t\to \infty} B_i(t)<1, \ i=1,2$, where
\[
\begin{split}
B_1(t):&=\frac{k(s)+G(t)}{a(t)}+\int_{t-\tau}^t(a(s)+k(s)+G(s))\, ds,\\
B_2(t):&=\frac{l(t)+H(t)}{b(t)}+\int_{t-\tau}^t(b(s)+l(s)+H(s))\, ds.
\end{split}
\]
Then the zero solution of  \eqref{f1-1} is stable and a global attractor of all solutions.
Moreover, under (ii)-(iii) and condition (i) above replaced by
 
(i') $\liminf_{t\to\infty} a(t)=\ul{a}>0$ and  $\liminf_{t\to\infty} b(t)=\ul{b}>0$,\\
then the zero solution of  \eqref{f1-1} is globally exponentially stable.
\end{thm}	

\begin{proof} The results on the  global attractivity and on the global exponential stability of the zero solution  follow
 immediately by applying Theorem 3.3 and Corollary 3.4 in \cite{FOJDDE26}, respectively,  to \eqref{f1-1}. 
\end{proof}

\subsection{Stability with zero hostility over time and incomplete information about the armament matrix}

We now consider system \eqref{f2-3} under the assumption of incomplete information about the armament matrix $A(t)$. Assuming that $A(t)\approx A$,  \eqref{f2-3} can be regarded
 as a perturbation of the autonomous linear DDE
 $u'(t)=Au(t-\tau)$ already studied in Section 2:
\begin{equation}\label{f2-5}
\frac{du}{dt}= Au(t-\tau)+\Delta A(t)u(t-\tau).
\end{equation}
Here, $A(t)=A+\Delta A(t)$, where $A$ is a matrix with constant elements and $\Delta A(t)= [\Delta a_{ij}(t)]$ is a $2\times 2$ matrix whose elements are continuous and satisfy
$$|\Delta a_{ij}(t)|\leq \bar{a}_{ij}, \ \ i,j = 1,2,$$
for some constant positive values $\bar{a}_{ij}$.
\begin{thm}\label{thm3.3} 
(a) If  the equilibrium $u=0$ of $u'(t)=Au(t-\tau)$ is   stable and there exists $\eta:\R_+\to \R$ with $\int_0^\infty \eta(t)\, dt<\infty$ such that
\begin{equation}\label{eta}\|\Delta A(t)\|\le \eta(t),\q t\ge 0,\end{equation}
then system \eqref{f2-5} is uniformly stable (on $[0,\infty)$).\\
(b) If  the equilibrium $u=0$ of $u'(t)=Au(t-\tau)$ is  EAS, then there exists 
 $\eta>0$ such that,  if
$|\Delta a_{ij}(t)|<\eta$ for $ i,j = 1,2$, 
 system \eqref{f2-5} is  EAS.
\end{thm}

\begin{proof} For general linear DDEs, it is well-known that  the zero solution is stable (respectively  asymptotically stable)  if and only if all its solutions are bounded on $[0,\infty)$ (respectively tend to zero as $t\to\infty$). 
Recall also that for the autonomous linear system $u'(t)=Au(t-\tau)$, the stability coincides with the uniform stability on $[0,\infty)$ and the asymptotic stability coincides with the exponential asymptotic stability  \cite{Hale77,HaleLunel}. From Theorem 5.1 of  \cite[Chapter  9]{Hale77}, if $u'(t)=Au(t-\tau)$ is  uniformly stable,  for perturbations sufficiently small satisfying \eqref{eta}, the perturbed system \eqref{f2-5} is uniformly stable as well.
Now, assume that $u'(t)=Au(t-\tau)$ is EAS. For the semiflow $\{T(t)\}_{t\ge 0}$ generated by this linear system,
there are constants $K>0, \al>0$ such that $\|T(t)\phi\|\le Ke^{-\al t}\|\phi\|$ for all $t\ge 0$ and $\phi\in C$. The result follows by standard arguments, similar to the ones used for ODEs,  using the estimate above,  the variation of constants formula and a Gronwall inequality. See \cite{Hale77} for more details.
\end{proof}

\begin{rmk}\label{rmk3.2}  {\rm
The above result (b)
does not provide much practical information, since it does not specify how small $\eta$ must be. One can however give an estimate for such $\eta$. In fact, as before, let $A=\left [\begin{matrix}-a&k\\l&-b\\ \end{matrix}\right]$ with $a,b,k,l>0$ and $\det A>0$, and  $\tau\in (0,\tau_{-})$, for $\tau_{-}$ defined in \eqref{0.1}, so that the system $u'(t)=Au(t-\tau)$ is EAS. Let $\al>0$ be such that all  characteristic roots $\la$ of this system satisfy $Re\, \la<-\al$. If $\eta\in (0,\al)$, then the usual arguments based on the variation of constants formula and  Gronwall inequality show that the perturbed system  \eqref{f2-5} is  EAS.}
\end{rmk}

\begin{rmk}\label{rmk3.2}  {\rm Using the perturbation theory in \cite[Chapter 9]{Hale77}, we can deduce results similar to the ones in Theorem \ref{thm3.3} for the local stability of the zero solution  of the original race arms system \eqref{f1-1}, assuming that $g(t, 0,0,0,0)=f(t, 0,0,0,0)=0$. Clearly, this setting can be applied to a positive equilibrium $P_0$ after a standard translation of $P_0$ to the origin.}
\end{rmk}

%


%
%

	\section{Asymptotic stability for a multi-dimensional equation $x'(t)=A(t)x(t-\tau)$ via special solutions}
	\setcounter{equation}{0}	
	
	Another approach to study \eqref{f2-3} includes Ryabov's concept  of ``special solutions" for linear DDEs, which has been exploited by several researchers since the pioneering works of Ryabov  \cite{Ryabov} and Driver \cite{Driver68,Driver}.	
\begin{defn} {\rm Consider a  linear DDE
\begin{equation}\label{linearDDE}
x'(t)=L(t)x_t,
\end{equation}
where $(t,\phi)\mapsto L(t)\phi\in \R^p$ is continuous on $\R\times C$ and 
 $L(t)$ is a continuous  linear operator from $C$ to $\R^p$, for all $t\in\R$.  For the abstract DDE  \eqref{linearDDE}, $C:=C([-\tau, 0];\R^p)$ is the phase space and, as usual, $x_t\in C$ is given by $x_t(\th)=x(t+\th)$ for $-\tau\le \th\le 0$.
 A {\bf special solution} of  \eqref{linearDDE} (if it exists)  is a solution defined on $\R$ such that $|x(t)|e^{t/\tau}$ is bounded for $t\le 0$. A {\bf special matrix solution} (if it exists) is a $p\times p$ matrix ${\cal X}(t)$ whose columns are special solutions of \eqref{linearDDE}  and such that ${\cal X}(0)=I$.}
\end{defn}



 Here, we shall consider a $p$-dimensional version of \eqref{f2-4}, for which
some important results established in Arino et al.  \cite{AGP}     are collected in the next lemma. See  also \cite{GPDIE97} and references therein.

\begin{lem}\label{lem3.1} Consider a fixed norm  on $\R^p$. Let $A(t)$ be a $p\times p$ matrix of continuous functions and assume that there are $m>0$ and $t_1\in\R$ such that
\begin{equation}\label{SSol_hyp}\sup_{t\ge t_1}\|A(t)\|\le m\q {\rm and}\q me\tau <1,\end{equation}
where $\|A(t)\|$ denotes de operator norm of the map $x\mapsto A(t)x$.
Then:\\
(a)  There exists a special matrix solution ${\cal X}(t)$ of\begin{equation}\label{linearDDE2}x'(t)=A(t)x(t-\tau);\end{equation} 
(b) For ${\cal X}(t)$, it holds: (i) its columns ${\cal X}_i(t)\, (1\le i\le p)$ satisfy
$|{\cal X}_i(t)|\le e^{\la_0(t-t_0)},\ t\ge t_0\ge 0,$
where $\la_0\in (-1/\tau,0)$ is the unique real solution of the equation
$me^{-\la \tau}=-\la;$
(ii) ${\cal X}(t)$ is non-singular for all $t\in\R$;
(iii) for each $\phi\in C$, the solution $x(t,0,\phi)$ of $x'(t)=A(t)x(t-\tau)$
 with initial condition $x_0=\phi$ satisfies
$$x(t,0,\phi)={\cal X}(t)\big [\ell (\phi)+o(1)\big]\q {\rm as}\q t\to\infty,$$
where  $\lim_{t\to\infty} {\cal X}^{-1}(t)X(t,0,\phi)=:\ell (\phi)\in\R^p;$
(iv) ${\cal X}(t)$ is a fundamental matrix solution of the ODE
$$x'=M(t)x,$$
where the $p\times p$ matrix of continuous functions $M(t)$ is defined by
\begin{equation}\label{M(t)}M(t)=\sum_{n=0}^\infty M_n(t,t),\q t\in\R,\end{equation}
 with the series in \eqref{M(t)} converging uniformly on $\R$,
for a sequence of matrix functions recursively given by
\begin{equation}\label{M_n}
\begin{split}
M_0(t,s)&=A(s),\q s\le t,\\
M_{n+1}(t,s)&=-A(s)\int_{s-\tau}^t M_n(t,u)\, du,\q s\le t, n\in \N_0.
\end{split}
\end{equation}
\end{lem}

\begin{rmk} 
\label{rmk3.3'}
{\rm In fact, the existence of a special matrix solution for \eqref{linearDDE2} follows under the  weaker assumption $\sup_{t\ge t_1}\int_{t-\tau}^t\|A(t)\|\, dt<1/e,$ rather than \eqref{SSol_hyp} (see \cite{GPDIE97}). We have however imposed the more restrictive constraint  \eqref{SSol_hyp} as in \cite{AGP}, since the reasoning pursued here requires  that the series in \eqref{M(t)} converges uniformly on $\R$, so that one derives that $M(t)$ is continuous. Moreover, condition \eqref{SSol_hyp} also allows us to use the
boundedness of $\|A(t)\|$ for $t>0$ large.}\end{rmk}

Under hypothesis \eqref{SSol_hyp}, it is therefore clear that  the zero solution  of \eqref{f2-4} is asymptotically stable if the same happens with the zero solution  of the ODE system $x'=M(t)x$, i.e., if
${\cal X}(t)\to 0$ as $t\to\infty$. Unfortunately, in general the properties of the  matrix $M(t)$ in \eqref{M(t)} are not easy to deduce, much less is it practicable to compute it, since $M(t)$ is given by the recursive scheme in \eqref{M_n}. Next,  we aim to derive the stability of $x'=M(t)x$ under some constraints. We start with an auxiliary statement.

\begin{lem}\label{lem3.2}  Consider \eqref{linearDDE2} where $A(t)$ is a $p\times p$ matrix of continuous functions.
Fix any norm on $\R^p$, and assume \eqref{SSol_hyp}. With the above notations,  write 
$$(-1)^n M_{n}(t,t)=\tau^{n}A^{n+1}(t)+A(t)\ga_n(t),\q n\in \N,$$
and $\Ga(t)=\sum_{n=1}^\infty (-1)^n\ga_n(t)$.
If 
 the ODE system 
 $$x'(t)=\left [A(t)(I+\tau A(t))^{-1}+A(t)\Ga(t)\right]x(t)$$
is GES,
then,  system \eqref{linearDDE2} is asymptotically stable.
\end{lem}

\begin{proof} First, observe that $\|\tau A(t)\| <1/e<1$ for $t\ge t_1$, for some $t_1\in\R$, thus the matrix  $(I+\tau A(t))^{-1}$ is well defined for $t\ge t_1$. Since
$$\sum_{n=0}^\infty (-1)^n\tau^{n}A^{n+1}(t)=A(t)\big(I+\tau A(t)\big)^{-1},$$
for $M(t)=\sum_{n=0}^\infty M_n(t,t)$ as in \eqref{M(t)}, this leads to 
$M(t)=A(t)\big(I+\tau A(t)\big)^{-1}+A(t)\Ga (t)$. If
the system $x'(t)=\left [A(t)(I+\tau A(t))^{-1}+A(t)\Ga(t)\right]x(t)$ is GES, 
 from Lemma \ref{lem3.1} we conclude that system \eqref{f2-4} is asymptotically stable, that is, the zero solution of \eqref{linearDDE2} is stable and a global attractor.
\end{proof}

\begin{thm}\label{thm4.1}  Fix any norm on $\R^p$, and consider \eqref{linearDDE2} where $A(t)=[a_{ij}(t)]$ is a $p\times p$ matrix of continuous functions.
Assume the following conditions:

(h1) there are $m>0$ and $t_1\in\R$ such that
\begin{equation}\label{SSol_A(t)}\sup_{t\ge t_1}\|A(t)\|\le m\q {\rm and}\q me\tau<1;\end{equation}

(h2) $\dps\lim_{t\to \infty}\Big (\max_{s\in[t-\tau,t]} \|A(t)-A(s)\|\Big)=0$;


(h3) for $\hat A(t)=\big [|a_{ij}(t)|\big ]$ and $ S_1:=\frac{m\tau}{1-m\tau}\log\left(\frac{1-m\tau}{1-2m\tau}\right)$, the ODE system $x'(t)=\big[A(t)(I+\tau A(t))^{-1}+S_1\hat A(t)\big]x(t)$ 
is GES.\vskip2mm

Then,  system \eqref{linearDDE2} is asymptotically stable.
\end{thm}

\begin{proof}As already observed in Lemma \ref{lem3.2}, the matrix $(I+\tau A(t))^{-1}$ is well defined. 


Fix an arbritary small $\vare \in (0,1) $, to be chosen later.
From (h2), there exists $t_2$ such that  $\max_{s\in[t-\tau,t]} \|A(t)-A(s)\|< \vare$ for $t\ge t_2$.
 Let $T_0=\max\{t_1,t_2\}+2\tau$. 

{\it Step 1}. For  $t\ge T_0, n\in\N$ and the matrices $M_n(t,t)$ as  in \eqref{M_n}, 
we first observe that
\begin{equation}\label{ind1}(-1)^n M_{n}(t,t)=\tau^{n}A^{n+1}(t)+A(t)\ga_n(t),\end{equation}
where 
\begin{equation}\label{ind2}
\begin{split}
&\ga_1(t)=\int_{t-\tau}^t(A(u)-A(t))\, du\\
&\ga_{n+1}(t)=\tau A(t) \ga_n(t)+\int_{t-\tau}^t (-1)^n[M_n(t,u)-M_n(t,t)]\, du,\q n\in\N.
\end{split}
\end{equation}
In fact, 
$$-M_1(t,t)=A(t)\int_{t-\tau}^t [A(t)+A(u)-A(t)]\, du=\tau A^2(t)+A(t)\ga_1(t),$$
for $\ga_1(t)$ as in \eqref{ind2}.
Assuming by induction that the  formula in  \eqref{ind1} is valid for some $n\in\N$, then 
\[ \begin{split}
(-1)^{n+1} M_{n+1}(t,t)&= A(t)\int_{t-\tau}^t  (-1)^n[M_n(t,t)+M_n(t,u)-M_n(t,t)]\, du\\
&=\tau A(t)[\tau^{n}A^{n+1}(t)+A(t)\ga_n(t)] + A(t)\int_{t-\tau}^t (-1)^n[M_n(t,u)-M_n(t,t)]\, du
\\
&=\tau^{n+1}A^{n+2}(t)+A(t)\ga_{n+1}(t)
\end{split}\]
for $\ga_{n+1}(t)=\tau A(t)  \ga_n(t)+  \int_{t-\tau}^t (-1)^n[M_n(t,u)-M_n(t,t)]\, du$. This proves \eqref{ind2}.

\med
{\it Step 2}. 
For $t\ge T_0$, $s\in [t-\tau,t]$   and $n\in \N$, we claim that
\begin{equation}\label{ind}
\begin{split}
&\|M_n(t,s)\|\le m^{n+1}(t-s+\tau)^n,\\
& \|M_n(t,t)-M_n(t,s)\|\le \vare m^n(t-s+\tau)^n+ m^{n+1}(t-s+\tau)^{n-1} (t-s).\\
\end{split}
\end{equation}

In fact, for $n=1$ we have:
\begin{equation}\label{ind_n=1}
\begin{split}
\|M_1(t,s)\|&\le \|A(s)\| \int_{s-\tau}^t\|A(u)\|\, du\le m^2(t-s+\tau)\q {\rm and}\\
 \|M_1(t,t)-M_1(t,s)\|&\le \left \| (A(t)-A(s))\int_{s-\tau}^tA(u)\, du\right\|+\left \|A(t) \int_{s-\tau}^{t-\tau}A(u)\, du\right\|\\
&\le \vare m(t-s+\tau)+ m^{2} (t-s).\\
\end{split}
\end{equation}
By induction, suppose now that \eqref{ind} is valid for some $n\in\N$.
Clearly we obtain 
$\|M_{n+1}(t,s)\|\le m\sup_{[t-2\tau,t]}\|M_n(t,s)\| (t-s+\tau)\le m^{n+2}(t-s+\tau)^{m+1}$ and
 \[ \begin{split}
 \|M_{n+1}(t,t)-M_{n+1}(t,s)\|&\le
 \left \| (A(t)-A(s))\int_{s-\tau}^tM_n(t,u)\, du\right\|+\left \|A(t) \int_{s-\tau}^{t-\tau}M_n(t,u)\, du\right\|\\
 &\le \vare m^{n+1}(t-s+\tau)^{n+1}+  m^{n+2}(t-s+\tau)^n (t-s).\\
 \end{split}\]

\med
{\it Step 3}.  For $t\ge T_0$ and $n\in\N$, we have
\begin{equation}\label{ind3}
\begin{split}
\|\ga_1(t)\|&\le \vare \tau\\
\|\ga_{n+1}(t)\|&\le m\tau \|\ga_n(t)\|+ \int_{t-\tau}^t \|M_n(t,u)-M_n(t,t)\|\, du\\
& \le m\tau \|\ga_n(t)\|+\vare m^{n} \int_{t-\tau}^t (t-u+\tau)^{n}\, du+m^{n+1}\tau \int_{t-\tau}^t (t-u+\tau)^{n-1}\, du\\
& =  m\tau \|\ga_n(t)\|+\vare \tau(m\tau)^{n} \Big(\frac{2^{n+1}-1}{n+1}\Big)+ (m\tau)^{n+1}\Big(\frac{2^n-1}{n}\Big).\\
\end{split}
\end{equation}
In particular,  by induction it is easily seen that for $n\ge 2$,
\begin{equation}\label{ga_n}\|\ga_{n}(t)\|\le \vare \tau  (m\tau)^{n-1} \sum_{k=1}^n\frac{2^k-1}{k}+(m\tau)^{n} \sum_{k=1}^{n-1}\frac{2^k-1}{k}.\end{equation}


Observe that $\sum_{k=1}^{n-1}\frac{2^k-1}{k}\le \sum_{k=1}^{n-1} 2^k=2(2^{n-1}-1)\le 2^n$. Since $2m\tau<2/e<1$,  the series
$\sum_{n=1}^\infty  (m\tau)^{n-1} \sum_{k=1}^n\frac{2^k-1}{k}$ and $\sum_{n=1}^\infty(m\tau)^{n} \sum_{k=1}^{n-1}\frac{2^k-1}{k}$ converge. Using standard Taylor series and logarithm properties, we obtain the sums
\[
\begin{split}
S_0:&=\sum_{n=1}^\infty  (m\tau)^{n-1} \Big(\sum_{k=1}^n\frac{2^k-1}{k}\Big)=\frac{1}{m\tau(1-m\tau)}\log\left(\frac{1-m\tau}{1-2m\tau}\right),\\
S_1:&=\sum_{n=1}^\infty  (m\tau)^n \Big(\sum_{k=1}^{n-1}\frac{2^k-1}{k}\Big)=\frac{m\tau}{1-m\tau}\log\left(\frac{1-m\tau}{1-2m\tau}\right).
\end{split}
\]
For $\Ga(t)=\sum_{n=1}^\infty (-1)^n\ga_n(t)$ and $t\ge T_0$, we have
$$\|\Ga(t)\|\le \sum_{n=1}^\infty \|\ga_{n}(t)\|\le  \vare \tau S_0+S_1.$$
For $M(t)=\sum_{n=0}^\infty M_n(t,t)$ as in \eqref{M(t)}, from Lemma \ref{lem3.2} this leads to
$$M(t)=A(t)\big(I+\tau A(t)\big)^{-1}+A(t)\Ga(t),$$
with $-S_1\hat A(t)+O(\vare)\le A(t)\Ga(t)\le S_1\hat A(t)+O(\vare)$.
From (h3),
system $x'(t)=\Big [A(t)(I+\tau A(t))^{-1}+S_1\hat A(t)\Big]x(t)$ is GES, and so it is any ODE system $x'(t)=\Big [A(t)(I+\tau A(t)]^{-1}+(S_1+ \de)\hat A(t)\Big]x(t)$ for $\de>0$ sufficiently small. For $\vare>0$ sufficiently small, we have  that the perturbed ODE system $x'(t)=\Big[A(t)(I+\tau A(t))^{-1}+(\vare \tau S_0+S_1)\hat A(t)\Big]x(t)$ is GES as well.
 From Lemma \ref{lem3.1}, we conclude that system \eqref{f2-4} is asymptotically stable, that is, the zero solution of \eqref{f2-4} is stable and a global attractor.
\end{proof}

\begin{rmk}\label{rmk3.4} {\rm Since $m\tau<1/e$, we always have the uniform bound $S_1<\frac{1}{e-1}\log \big(\frac{e-1}{e-2}\big)\approx 0.51$; nevertheless,  we stress that $S_1=S_1(\tau)$ in (h3) is such that   $S_1(\tau)\to 0$ as $\tau\to 0$. Note also that  the estimates in Steps 2 and  3 are not optimal, but they  are  easier to deal with than sharper ones. For instance,  we could have used integration by parts in \eqref{ind3},  to compute $\int_{t-\tau}^t (t-u+\tau)^{n-1}(t-u)\, du= \tau^{n}\big[\frac{2^n}{n}+\tau \frac{1-2^n}{n(n+1)}\big]$, which however is not very helpful  to derive a recursive formula as in \eqref{ga_n}.
On the other hand, under some additional constraints, one can obtain better estimates for the norm of the errors $\ga_n(t)$, and therefore also for the sum $S_1$, or even show that $\|\ga_n(t)\|={\rm O}(\vare)$, for any $\vare>0$ fixed.
}
\end{rmk}

In particular, the above results can be applied to the linear approximation \eqref{f2-4} of the  arms race  model. The underlying
threshold condition $m\tau < 1/e$ (to derive a possible stability criterion) can be interpreted as follows: if the product of the response intensity and the delay time exceeds the threshold  $1/e \approx 0.37$, the system risks getting out of control. With large delays, even a moderate response can lead to oscillations and an uncontrolled increase in tension.

As an illustration,  in the next example the results above are applied to the confrontation  model \eqref{f2-4} with $A(t)$ a triangular matrix.

\begin{exmp}\label{exmp3.1} {\rm
Fix e.g.  the maximum norm in $\R^2$.  Let
$A(t)=\left [\begin{matrix}-a(t)&k(t)\\0&-b(t)\\ \end{matrix}\right]$
with $a(t),b(t), k(t)$  positive and continuous, and suppose that
assumptions (h1),(h2) in Theorem \ref{thm4.1} are satisfied, where now $l(t)\equiv 0$. 
In particular, we require that
$$\tau \sup_{t\ge t_1}\|A(t)\|=\tau \sup_{t\ge t_1}\max\{a(t)+k(t),b(t)\}<1/e.$$

%
%
For the present situation, simple computations lead to
\[
A(t)[I+\tau A(t)]^{-1}=\left [\begin{matrix} -\frac{a(t)}{1-\tau a(t)}&\frac{k(t)}{(1-\tau a(t))(1-\tau b(t))}\\ 0&-\frac{b(t)}{1-\tau b(t)}\end{matrix}\right],
\]

 Now, consider the notations as well as the  computations in the proof of Theorem \ref{thm4.1}. For $\vare>0$ small and $t>0$ large, these estimates yield
$$M(t)=
\left [\begin{matrix} -\frac{a(t)}{1-\tau a(t)}&\frac{k(t)}{(1-\tau a(t))(1-\tau b(t))}\\ 0&-\frac{b(t)}{1-\tau b(t)}\end{matrix}\right]+A(t)\sum_{n=1}^\infty (-1)^n \ga_n(t),
$$
with $\sum_{n=1}^\infty\| \ga_n(t)\|\le S_1\le \frac{m\tau}{1-m\tau}\log\left(\frac{1-m\tau}{1-2m\tau}\right)$ for $m=\sup_{t\ge t_1}\|A(t)\|$.
Consider the matrix 
$$M^*(t)=\left [\begin{matrix} -a(t)(\frac{1}{1-\tau a(t)}-S_1)&k(t)\Big(\frac{1}{(1-\tau a(t))(1-\tau b(t))}+S_1\Big)\\ 0&-b(t)(\frac{1}{1-\tau b(t)}-S_1)\end{matrix}\right].$$
If
\[
\begin{split}
&\liminf_{t\to\infty} b(t)\big(\frac{1}{1-\tau b(t)}-S_1\big)>0\\
 &\liminf_{t\to\infty}\left(\frac{a(t)}{1-\tau a(t)}-\frac{k(t)}{(1-\tau a(t))(1-\tau b(t))}-S_1(a(t)+k(t))\right)>0,
\end{split}
\]
then  $\limsup_{t\to\infty}M^*(t){\bf 1}<0$, for the matrix vector ${\bf 1}=[1\ 1]^T$. From a criterion in \cite[Proposition 6.3]{Coppel}, it follows that $x'=M^*(t)x$ is GES, and so it is $x'=[A(t)(I+\tau A(t))^{-1}+S_1\hat A(t)]x$. From Theorem \ref{thm4.1}, system  \eqref{f2-4} is asymptotically stable.

For instance, let $a(t)=b(t)\equiv a>0$. In this case, for both the norms $|\cdot|_1$ and $|\cdot|_\infty$ in $\R^2$, $\|A(t)\|=a+k(t)$.  Suppose that, for some $t_1\in\R$,
$$\tau \sup_{t\ge t_1}k(t)<1/e-\tau a,$$
and that $\kappa(t):=\max\{|k(s)-k(t)|: s\in [t-\tau,t]\}=o(1)$ as $t\to\infty$.
Hence, hypotheses (h1) and (h2) hold. On the other hand,   
$ A(t)[I+\tau A(t)]^{-1}= \left [\begin{matrix} -\frac{a}{1-\tau a}&\frac{k(t)}{(1-\tau a)^2}\Big)\\ 0&-\frac{a}{1-\tau a}\end{matrix}\right].$  Since $a(t),b(t)$ are constant, one easily checks by induction that the matrices $A(t)\ga_n(t)$ have de form $A(t)\ga_n(t)= \left [\begin{matrix} 0&k(t)\ga_{12}^{(n)}(t)\\ 0&0\end{matrix}\right]$, so $A(t)\Ga(t)= \left [\begin{matrix} 0&(-1)^nk(t)\Ga_{12}(t)\\ 0&0\end{matrix}\right]$, with $\|k(t)\Ga_{12}(t)\|\le  S_1|k(t)|$.  
In conclusion, if in addition  $$\sup_{t\ge t_2} k(t)<a(1-\tau a)[1+S_1(1-\tau a)^2]^{-1}$$ for some $t_2$, then  it follows that $\limsup_{t\to\infty} [A(t)(I+\tau A(t))^{-1}+A(t)\Ga(t)]{\bf 1}<0$, 
and Lemma \ref{lem3.2} implies that  system 
$x'(t)=\left [\begin{matrix}-a&k(t)\\0&-a\\ \end{matrix}\right]x(t-\tau)$ is asymptotically stable. 
Taking the Euclidean norm in $\R^2$, similar conclusions can be drawn
using Proposition \ref{prop3.1}.
}
\end{exmp}

\section{Concluding Remarks}
	\setcounter{equation}{0}

In this article, we propose an original  delayed Richardson model for an arms race. With the introduction of a delay $\tau>0$, representing a reaction time for armament and in the hostility factors, the model becomes more realistic.  A time-delay had already been considered by Hill \cite{HILL}, but only  for the  linear autonomous arms race system \eqref{Hill}, a simplified version of the model \eqref{f1-1} proposed and studied here. Even for this very particular case, we  carefully analysed its characteristic equation and provided a correction on the value of the critical threshold $\tau_{-}$ computed in  \cite{HILL}, which assures  stability if and only if $\tau<\tau_-$. In summary, for the autonomous model, we have derived sharp conditions for the asymptotic stability of the equilibrium and for  the occurrence of  Hopf bifurcations, as the delay increases and crosses some critical values $\tau_n^{\pm}$, with the direction and stability of the bifurcating periodic orbits completely described under very natural conditions. For the non-autonomous model, several criteria for stability, asymptotic stability and exponential asymptotic stability were given, by using several different techniques. In addition, a general result on the asymptotic stability for  $p$-dimensional DDEs of the form $x'(t)=A(t)x(t-\tau)$ was established by using  the theory of the so-called {\it special solutions}, which in turn,  as a particular case of such DDEs with $p=2$, can be applied to the linear approximation  for \eqref{f1-1} at an equilibrium point.

\med

Below,  a possible political and strategic interpretation of the results  concerning the Hopf bifurcations is presented.

\med

1. Formula for $K_{1}$ and ``the inevitability of destabilization": We have proven that $K_{1}>0$ always.

Interpretation: As soon as the decision-making time-delay  $(\tau)$ exceeds the critical threshold $\tau_{-}$, the equilibrium (``peace") is guaranteed to lose stability. No amount of good intentions on the part of the parties can prevent this if institutional inertia (bureaucracy, intelligence, logistics) is too great.

2. Formula for $K_{2}$ and ``the anatomy of a global crisis": This is a key point of our work. The sign of $K_{2}$ determines what will happen to the countries in conflict after the stability threshold is crossed. The following are possible:

-- "Soft" crisis scenario ($K_{2}<0$, supercritical bifurcation): This creates stable cycles.

Interpretation: This is the ``chronic arms race" model. The system does not explode into infinity (toward total war), but it also does not return to peace. Countries are locked in a cyclical process: increasing power $\rightarrow$ economic peak $\rightarrow$ fatigue/detente $\rightarrow$ new fear $\rightarrow$ new cycle.

The world becomes predictably unstable. This is  a ``controlled chaos," where the arms race becomes a constant backdrop to existence.

-- ``Hard" crisis scenario ($K_2>0$, subcritical bifurcation): This creates unstable periodic orbits.

Interpretation: This is the `powder keg" model. The equilibrium point becomes extremely vulnerable. Any accidental threshold exceedance (technical failure, provocation) instantly throws the system out of control. Decisions go ``off the rails"   towards exponential escalation.

This is a harbinger of a catastrophic collision. Here, the delayed decisions $(\tau)$ act as a detonator.

3. The role of nonlinear aggression $(g_{ij},h_{ij})$: Under (H2), the formula for $K_{2}$ shows that the outcome depends on the cubic nonlinearity of hostility.

Specifically, the fate of the world depends on how the level of ``hatred" (propaganda) changes with large quantities of weapons.
If, upon reaching huge arsenals, society begins to get tired from hatred (the growth of hostility slows), then $K_{2}$
will likely be negative, and a stable cycle will be achieved. If propaganda operates on the principle of ``the more missiles, the more hatred" (accelerating growth of aggression), then $K_{2}$
becomes positive, and the system is doomed to an explosive scenario.

4. Multiple bifurcation points $(\tau_{n}^{\pm})$: The fact that Theorems \ref{thm2.1} and \ref{thm2.2} apply to any $\tau_{n}^{\pm}$ indicates that the system has multiple ``lines of defense" (or ``lines of fall").

As the delay $\tau$ increases, the system passes through a series of turbulent zones. Even if countries survive the first crisis $\tau_{0}^{-}$, they may face the next ($\tau_{0}^{+}$ or $\tau_{1}^{-}$), where the oscillation structure becomes even more complex.

\med

As lines of future research on this arms race model, some questions to be explored follow:

--- How does the ``propaganda ceiling" ($\overline{b}$) affect the ability to stabilize an arms race? If propaganda is too strong, can it overcome economic constraints ($a,b$)?

It is of interest to follow the mathematical conditions linking the parameters of the model $(k,a,l,b,\tau)$. 

--- How do the threat-cost ratios $(k/a)$ and $(l/b)$ limit the permissible delay time $(\tau)$?

--- Is the hypothesis true that the more aggressive countries are $(k, l)$, the less they have the right to``think slowly" (i.e., $\tau$
must be very small to avoid catastrophe)?

 \section*{Acknowledgement} 
 The work of T. Faria  was  financed by Portuguese funds through FCT (Funda\c c\~ao para a Ci\^encia e a Tecnologia), within the Project UID/04561/2025 - https://doi.org/10.54499/UID/04561/2025.

%


	\end{document}